\documentclass[11pt, a4paper]{amsart}
\usepackage[margin=0.9in]{geometry}
\usepackage{amsfonts}   
\usepackage{amsmath}
\usepackage[utf8]{inputenc}
\usepackage[english]{babel}
\usepackage{amssymb}
\usepackage{tikz-cd}
\usepackage{comment}
\usepackage{amsthm}
\usepackage{xcolor}
\usepackage{float}
\usepackage[normalem]{ulem}
\newtheorem{theorem}{Theorem}[section]
\newtheorem{cor}{Corollary}[section]

\newtheorem{lemma}{Lemma}[section]

\newtheorem*{claim*}{Claim}
\newtheorem*{proofclaim*}{Proof of Claim}

\theoremstyle{theorem}

\newtheorem*{remark}{Remark}
\usepackage{embedfile}

\newtheorem*{Exercise 11.3.B}{Exercise 11.3.B}

\numberwithin{equation}{section}

\newcommand{\p}{\textbf{Proof. }}

\newcommand{\n}{\mathfrak{n}}

\renewcommand{\pmod}[1]{\,(\textup{mod}\,#1)}

\usepackage{fancyhdr}

\usepackage{graphicx}
\graphicspath{ {./images/} }

\usepackage{hyperref}                                    
\usepackage{enumerate}

\usepackage[all,cmtip]{xy}

\usepackage{blindtext}

\begin{document}
\title[ Primes \( p \) such that \( p - b \) Has a Large Power Factor and Few Other Prime Divisors]
	{  Primes \( p \) such that \( p - b \) Has a Large Power Factor and Few Other Prime Divisors}
	\author{Likun Xie}
	\address{Department of Mathematics, University of Illinois, 1409 West Green
		Street, Urbana, IL 61801, USA} 
	\email{likunx2@illinois.edu}

\begin{abstract} 
We prove lower bounds for the number of primes \( p \leq N + b \) such that \( p - b \) is divisible by \( 2^{k(N)} \) and has at most \( k \) odd prime factors (\( k \geq 2 \)), assuming \( 2^{k(N)} \leq N^\theta \) for some \( \theta > 0 \) depending on \( k \). The proof uses a variant of Chen's method, weighted sieves, and Elliott's results on primes in arithmetic progressions with large power-factor moduli.

\end{abstract}

	\maketitle

	\section{Introduction}\label{sec: Introduction}
Our motivating question arises from a problem posed by Hooley \cite[p.~109]{hooley}, concerning the number of primes in sequences of the form \( a^n + b \), or equivalently, the number of primes \( p \leq x \) such that
\[
p - b = a^n.
\]
In particular, Hooley presents the following result \cite[p.~113]{hooley}.

\begin{theorem}
	Let \( b > 1 \) be an odd integer, and define \( \pi_b(x) \) to be the number of primes of the form \( 2^n + b \) with \( n \leq x \). Then
	\[
	\pi_b(x) = o(x)
	\]
	provided both the Extended Riemann Hypothesis and Hypothesis~A \cite[p.~112]{hooley} hold.
\end{theorem}

	While Hooley's original problem remains unsolved, we prove several related results concerning the existence of infinitely many primes \( p \) such that 
	\[
	p - b \text{ is divisible by a large  power factor } a^n 
	\]
	and has very few other divisors other than $a$. 
For simplicity, we consider the case where \( b \) is odd and \( a = 2 \), though the argument extends to any integer \( a \geq 2 \). We use an elementary variant of Chen’s theorem \cite{chen} and linear weighted sieve methods \cite{jurkat-Richert, halberstam}. In place of the classical Bombieri--Vinogradov theorem, we employ Elliott’s results \cite{prime_power} on the distribution of primes in arithmetic progressions with large power-factor moduli.

\begin{theorem}[{\cite[Theorem]{prime_power}}]\label{elliott}
	Let \( a \geq 2 \) be an integer. For any \( A > 0 \), there exists a constant \( B > 0 \) such that
	\[
	\sum_{\substack{d \leq x^{1/2} q^{-1} (\log x)^{-B} \\ (d, q) = 1}} \max_{\substack{(r, qd) = 1}} \max_{y \leq x} \left| \pi(y, qd, r) - \frac{\operatorname{li}(y)}{\varphi(qd)} \right| \ll \frac{x}{\varphi(q)(\log x)^A}
	\]
	uniformly for moduli \( q \leq x^{1/3} \exp\left( -(\log \log x)^3 \right) \) that are powers of \( a \).
	
	The implied constant may depend on \( a \), while \( B \) depends only on \( A \); in particular, one may take \( B = A + 6 \).
\end{theorem}

We briefly explain the origin of the bound on the modulus \( q \).
For any modulus \( D \), \( 0 < \alpha \leq 1 \), and \( T \geq 0 \), let \( N(\alpha, T, D) \) denote the number of zeros, counted with multiplicity, of all functions \( L(s, \chi) \) formed with characters \( \chi \pmod{D} \) that lie in the rectangle
\[
\alpha \leq \Re(s) \leq 1, \quad |\Im(s)| \leq T.
\]
The bound for the modulus \( q \leq x^{1/3} \exp\left( -(\log \log x)^3 \right) \) essentially follows from the following upper bound on \( N(\alpha, T, D) \):

\begin{lemma}[{\cite[Lemma 5]{prime_power}}]\label{zero_lemma}
	There are positive constants \( \lambda < 3 \), \( c_2 \), and \( c_3 \) such that
	\[
	N(\alpha, T, D) \leq c_2 (D T)^{\lambda(1 - \alpha)} (\log D T)^{c_3}
	\]
	uniformly for \( 0 \leq \alpha \leq 1 \) and \( T \geq 2 \).
\end{lemma}
As noted in \cite{prime_power}, one may   use different versions of this type of estimate—for instance, Montgomery's version \cite[p.~98]{montgomery}, which gives \( \lambda = \frac{5}{2} \) and \( c_3 = 14 \). Improvements include a result of Jutila \cite{jutlia}, who obtained \( \lambda = \frac{3(9 + \sqrt{17})}{16} \), and a version by Heath-Brown \cite{heathbrown}, which allows \( c_3 = 0 \) and any fixed \( \lambda > \frac{12}{5} \).
 Following the same method as in Elliott \cite{prime_power} leads to various upper bounds for \( q \).

Here, we adopt the version with \( \lambda < \frac{5}{2} \), which yields the modulus range \( q \leq x^{2/5} \exp\left( -(\log \log x)^3 \right) \), and apply it in our error estimates for the sieve.\footnote{This result appears in an article published in \emph{Advances in Pure Mathematics}, which we do not cite due to concerns about the journal’s reliability. As noted above, the bound essentially follows from Elliott \cite{prime_power}, using Lemma~\ref{zero_lemma} with \( \lambda < \frac{5}{2} \).}

\begin{theorem} \cite{prime_power}\label{guo}
Let \( a \geq 2 \) be an integer. If \( A > 0 \), then there exists a constant \( B = B(A) > 0 \) such that
\[
\sum_{\substack{d \leq x^{1/2} q^{-1} (\log x)^{-B} \\ (d, q) = 1}} \max_{\substack{(r, qd) = 1}} \max_{y \leq x} \left| \pi(y, qd, r) - \frac{\operatorname{li}(y)}{\varphi(qd)} \right| \ll \frac{x}{\varphi(q)(\log x)^A}
\]
holds uniformly for moduli \( q \leq x^{2/5} \exp\left(-(\log \log x)^3\right) \) that are powers of \( a \).

\end{theorem}

An application of a linear weighted sieve yields the following result, proved in Section~\ref{section_weighted_sieve}.
\begin{theorem}\label{intro_theorem_s3}
Let \( b > 0 \) be an odd integer, and let \( \epsilon > 0 \) be a small fixed constant. Suppose \( k: \mathbb{N} \to \mathbb{N} \) is an arithmetic function satisfying
\[
2^{k(N)} \leq N^{a - \epsilon}, \quad \text{where } a = \frac{1}{6}.
\]

Let \( s_3(N  ) \) denote the number of primes \( p \leq N + b \) such that \( p - b \) is divisible by \( 2^{k(N)} \) and has at most three odd prime factors.

Define \( N' := N / 2^{k(N)} \). Then for sufficiently large $N$,  
\[s_3(N   ) \geq \frac{40\theta}{\left( \frac{2}{5} + 6\theta \right)^2} \log 3 \cdot \prod_{p > 2} \left(1 - \frac{1}{(p - 1)^2} \right) \prod_{2 < p \mid b} \frac{p - 1}{p - 2} \cdot \frac{N'}{(\log N)(\log N')},\]
where
\[
\theta = \frac{\epsilon}{2(1 - a + \epsilon)(1 - a)}.
\]

\end{theorem}

By applying similar methods and adjusting the sieve parameters, we obtain the following corollary, which gives results under various upper bounds on \( 2^{k(N)} \).

	Let \( s_K(N) \) denote the number of primes \( p \leq N + b \) such that \(p-b\) is divisible by $2^{k(N)}$, 
and has at most \( K \) odd prime factors.
\begin{cor}\label{intro_cor_sk}
Let \( b > 0 \) be an odd integer, and let \( \epsilon > 0 \) be a small fixed constant. Suppose \( k: \mathbb{N} \to \mathbb{N} \) is an arithmetic function satisfying
\[
2^{k(N)} \leq N^{a - \epsilon}, \quad \text{for some } 0<a <1.
\]

Define \( N' := N / 2^{k(N)} \). Then for \( K = 4, 5, 6, 7 \), and all sufficiently large \( N \), we have
\[	s_K(N  ) \geq \frac{8\theta(1 + 3K)^3\log 3 }{\left[8 + 3\theta(1 + 4K + 3K^2)\right]^2}   \prod_{p > 2} \left(1 - \frac{1}{(p - 1)^2}\right) \prod_{2 < p \mid b} \frac{p - 1}{p - 2} \cdot \frac{N'}{(\log N)(\log N')},\]
where
\[
\theta = \frac{\epsilon}{2(1 - a + \epsilon)(1 - a)}, \quad \text{and} \quad a = \frac{3K - 7}{6K - 6}.
\]
The corresponding values of \( a \) for each \( K \) are given below:

\begin{center}
	\begin{tabular}{|c|c|c|c|c|}
		\hline
		\( K \) & 4 & 5 & 6 & 7 \\
		\hline
		\( a \) & \( \dfrac{5}{18} \) & \( \dfrac{1}{3} \) & \( \dfrac{11}{30} \) & \( \dfrac{7}{18} \) \\
		\hline
	\end{tabular}
\end{center}

For \( K = 8 \), suppose instead that
\[
2^{k(N)} \leq N^{2/5} \exp\left(-(\log \log N)^3\right),
\]
then for all sufficiently large \( N \), we have
\[	s_8(N  ) \geq \frac{52}{3} \cdot \prod_{p > 2} \left(1 - \frac{1}{(p - 1)^2}\right) \prod_{2 < p \mid b} \frac{p - 1}{p - 2} \cdot \frac{N'}{(\log N)(\log N')}.\]

\end{cor}

The number of odd prime divisors can be further reduced to two by adapting Chen's method. There are two natural sieving approaches.

The first is to sieve the sequence $\frac{p - b}{2^{k(N)}}$, where $p - b$ is divisible exactly by $2^{k(N)}$. This leads to the same weight function used in Chen’s original method and counts primes $p \leq N + b$ such that $p - b = 2^{k(N)} P_2$, where $P_2$ denotes an integer with at most two prime factors.

The second approach allows additional powers of 2 to divide $\frac{p - b}{2^{k(N)}}$, sieving over   sequence $\frac{p - b}{2^{k(N)}}$ where $p - b$ is divisible by $2^{k(N)}$, but not necessarily exactly. In this case, we count primes $p \leq N + b$ such that $p - b = 2^{k(N) + m} P_2$ for some $m \geq 0$, and the weight function must be appropriately modified.

The first approach leads to the following result.

	\begin{theorem}\label{intro_theorem_n2}
	Let \( b > 0 \) be an odd integer, and let \( k: \mathbb{N} \to \mathbb{N} \) be an arithmetic  function such that for some constant \( 0 < c < 1 \), we have
\[
cN^a < 2^{k(N)} \leq N^a , \quad \text{where } a = \frac{1}{87}.
\]

Define \( N' := N / 2^{k(N)} \), and let \( n_2(N ) \) denote the number of primes \( p \leq N + b \) such that
\[
p - b = 2^{k(N)} P_2\left(N'^{85/688}\right),
\]
where \( P_r(z) \) denotes a number that is a product of at most \( r \)   primes, all greater than \( z \).
 Then for sufficiently large $N$, 
\[	n_2(N ) \geq 0.0016 \prod_{p > 2} \left(1 - \frac{1}{(p - 1)^2} \right) \prod_{2 < p \mid b} \frac{p - 1}{p - 2} \cdot \frac{N'}{(\log N')(\log N)}.\]

\end{theorem}
The second approach yields the following result. Note that the constant in the lower bound for \( s_2(N) \) is exactly twice that of \( n_2(N) \) under the same condition on \( 2^{k(N)} \).

	\begin{theorem}
	Let \( b > 0 \) be an odd integer, and let \( k: \mathbb{N} \to \mathbb{N} \) be an arithmetic  function such that for some constant \( 0 < c < 1 \), we have
\[
cN^a < 2^{k(N)} \leq N^a , \quad \text{where } a = \frac{1}{87}.
\]

	Define \( N' := N / 2^{k(N)} \). Then for sufficiently large $N$, 
	\[	s_2(N ) \geq 0.0032 \prod_{p > 2} \left(1 - \frac{1}{(p - 1)^2} \right) \prod_{2 < p \mid b} \frac{p - 1}{p - 2} \cdot \frac{N'}{(\log N)(\log N')}.\]
	\end{theorem}

\subsection*{Notation}

Throughout this paper, \( p, p_1, p_2, p_3 \), and \( q \) denote prime numbers.  We write \( \varphi(n) \) for Euler’s totient function, \( \mu(n) \) for the M\"obius function, \( \nu(n) \) for the number of distinct prime divisors of \( n \), and \( d(n) \) for the number of divisors of \( n \). The logarithmic integral is denoted by \( \operatorname{li}(x) = \int_2^x \frac{dt}{\log t} \).

\section{\texorpdfstring{$p - b$ divisible   by $2^{k(N)}$ and at most two odd prime divisors}{p - b divisible exactly by 2^k(N) and at most two odd prime divisors}}

Chen’s theorem \cite{chen} states that every sufficiently large even integer \( N \) can be written as
\[
N = p + P_2,
\]
where \( P_2 \) denotes an almost prime of order 2, i.e., a product of at most two  primes. A similar argument shows that for any \( k \geq 1 \), there are infinitely many primes \( p \) such that \( p + 2k = P_2 \).

In this section, we adapt Chen’s method to show that there are infinitely many primes \( p \) such that \( p - b \) is divisible by a large power of \( 2 \) and has at most two odd prime factors. Our approach follows the framework of \cite[Chapter~10]{nathanson}, beginning with the linear sieve as presented in \cite[Chapter~11]{opera}; see also \cite[Theorem~9.7]{nathanson}.

Let \( A = \{a(n)\}_{n=1}^\infty \) be a sequence of non-negative real numbers such that
\[
|A| := \sum_{n=1}^\infty a(n) < \infty.
\]
When referring to a set \( A \), we identify it with its characteristic function, defined by
\[
a(n) =
\begin{cases}
	m_n, & \text{if } n \in A \text{ with multiplicity } m_n, \\
	0, & \text{otherwise}.
\end{cases}
\]
Let \( \mathcal{P} \) be a set of primes, and for \( z \geq 2 \), define
\[
P(z) := \prod_{\substack{p \in \mathcal{P} \\ p < z}} p, \quad 
S(A, \mathcal{P}, z) := \sum_{\substack{n \geq 1 \\ (n, P(z)) = 1}} a(n).
\]

Suppose \( g(d) \) is a multiplicative function such that
\[
0 \leq g(p) < 1 \quad \text{for all } p \in \mathcal{P},
\]
and that
\begin{equation}\label{condition_epsilon}
	\prod_{\substack{w \leq p < z \\ p \in \mathcal{P}}} \left(1 - g(p)\right)^{-1}
	\leq \frac{\log z}{\log w} \left( 1 + \frac{L}{\log w} \right)
\end{equation}
for every \( 2 \leq w < z \), where \( L \geq 1 \) is a constant.

Define
\[
X := |A| = \sum_{n=1}^\infty a(n), \quad \text{and} \quad V(z) := \prod_{p \mid P(z)} (1 - g(p)).
\]
For each \( d \mid P(z) \), define
\[
|A_d| := \sum_{\substack{n \geq 1 \\ d \mid n}} a(n) = g(d) X + r(d),
\]
where \( r(d) \) is the remainder term.

We state the following theorem for the linear sieve. The original result is due to Jurkat and Richert \cite{jurkat-Richert}; see also \cite[Theorem~11.6]{opera} for a more general formulation applicable to sieves of various dimensions. Although the derivations differ slightly, the upper and lower bound functions are the same. Here, we follow the formulation in \cite{opera}.

\begin{theorem}\cite[Theorem 11.6]{opera}\label{Jurkat}
	Let \[s = \frac{\log D}{\log z} .\] Then, if \( s \geq 1 \), we have the upper bound
	\[
	S(A, \mathcal{P}, z) \leq X V(z) \left(F(s) + O\left((\log D)^{-1/6}\right)\right) + O\left(\sum_{\substack{d \mid P(z) \\ d < D}} |r(d)|\right),
	\]
	and if \( s \geq 2 \), the lower bound
	\[
	S(A, \mathcal{P}, z) \geq X V(z) \left(f(s) + O\left((\log D)^{-1/6}\right)\right) + O\left(\sum_{\substack{d \mid P(z) \\ d < D}} |r(d)|\right),
	\]
	where \( F(s) \) and \( f(s) \) are the standard upper and lower bound sieve functions as defined in  \cite[(12.1)--(12.2)]{opera}; see also \cite[(9.27)--(9.28)]{nathanson}.
\end{theorem}
\begin{lemma}\label{initial_value}
\[
F(s) = \frac{2e^\gamma}{s}, \quad \text{for } 1 \leq s \leq 3,
\]
\[
f(s) = \frac{2e^\gamma \log(s - 1)}{s}, \quad \text{for } 1 \leq s \leq 4,
\]
where \( \gamma \) is Euler’s constant.
\end{lemma}

\begin{proof}
	This follows directly from \cite[(12.1)--(12.2)]{opera} or from \cite[Theorem 9.8]{nathanson}.
\end{proof}

	\begin{theorem}\label{main_theorem}
		Let \( b > 0 \) be an odd integer, and let \( k: \mathbb{N} \to \mathbb{N} \) be an arithmetic  function such that for some constant \( 0 < c < 1 \), we have
	\[
	cN^a < 2^{k(N)} \leq N^a , \quad \text{where } a = \frac{1}{87}.
	\]
	
	Define \( N' := N / 2^{k(N)} \), and let \( n_2(N ) \) denote the number of primes \( p \leq N + b \) such that
	\[
	p - b = 2^{k(N)} P_2\left(N'^{85/688}\right),
	\]
	where \( P_r(z) \) denotes a number that is a product of at most \( r \)   primes, all greater than \( z \).
	Then for sufficiently large $N$, 
	\[	n_2(N ) \geq 0.0016 \prod_{p > 2} \left(1 - \frac{1}{(p - 1)^2} \right) \prod_{2 < p \mid b} \frac{p - 1}{p - 2} \cdot \frac{N'}{(\log N')(\log N)}.\]

	\end{theorem}

	\begin{proof}
Throughout the proof, let \( p \) and \( q \) denote primes. Define
\begin{equation}\label{definition_y}
	N' = \frac{N}{2^{k(N)}}, \quad 
	y = N'^{1/3}, \quad 
	z = N'^{\delta},
\end{equation}
where \( 0 < \delta < \frac{1}{3} \) is a parameter to be chosen later.

We define the weight function
\begin{equation}\label{n_2_weight}
	w(n) = 1 
	- \frac{1}{2} \sum_{\substack{z \leq q < y \\ q^k \,\|\, n}} k
	- \frac{1}{2} \sum_{\substack{p_1 p_2 p_3 = n \\ z \leq p_1 < y \leq p_2 \leq p_3}} 1.
\end{equation}

Let \( \mathcal{P} \) denote the set of primes \( p > 2 \) with \( (p, b) = 1 \), and define
\[
P(z) := \prod_{\substack{p \in \mathcal{P} \\ p < z}} p = \prod_{\substack{(p, 2b) = 1 \\ p < z}} p.
\]

Consider the set
\begin{equation}\label{definition_A_set}
	A := \left\{ \frac{p - b}{2^{k(N)}} : 0 < p - b \leq N,\ 2^{k(N)} \,\|\, (p - b) \right\},
\end{equation}
and let \( A_d \) denote the subset of \( A \) consisting of elements divisible by \( d \). Then,
\[
|A| = \pi(N + b, 2^{k(N)}, b) - \pi(N + b, 2^{k(N)+1}, b).
\]
For each \( n = \frac{p - b}{2^{k(N)}} \in A \), we have \( (n, 2b) = 1 \). Indeed, if some odd prime \( q \mid (n, b) \), then \( q \mid p \), which implies \( p = q \leq b \), contradicting \( p > b \).
	
If \( n \leq N' \), \( (n, 2b) = (n, P(z)) = 1 \), and \( w(n) > 0 \), then \( n \) is either \( 1 \), a prime \( p_1 \geq z \), or a product \( p_1 p_2 \), where \( p_1, p_2 \) are primes satisfying \( z \leq p_1 < y \leq p_2 \).
Define
\[
S := \left\{ 1,\ p_1,\ p_1 p_2 : p_1, p_2 \geq z,\ \text{both } p_1 \text{ and } p_2 \text{ prime} \right\}.
\]

Then we have
\begin{align*}
	n_2(N) &\geq \sum_{\substack{n \in A \\ n \in S \\ (n, P(z)) = 1}} 1 
	\geq \sum_{\substack{n \in A \\ n \in S \\ (n, P(z)) = 1}} w(n) 
	\geq \sum_{\substack{n \in A \\ (n, P(z)) = 1}} w(n) \\
	&= \sum_{\substack{n \in A \\ (n, P(z)) = 1}} 1 
	- \frac{1}{2} \sum_{\substack{n \in A \\ (n, P(z)) = 1}} 
	\sum_{\substack{z \leq q < y \\ q^k \parallel n}} k 
	- \frac{1}{2} \sum_{\substack{n \in A \\ (n, P(z)) = 1}} 
	\sum_{\substack{p_1 p_2 p_3 = n \\ z \leq p_1 < y \leq p_2 \leq p_3}} 1.
\end{align*}
Now, we express these three sums as sieving functions.

The first sum becomes
\[
\sum_{\substack{n \in A \\ (n, P(z)) = 1}} 1   = S(A, \mathcal{P}, z).
\]

Next, we split the second sum into two parts 
\begin{equation}\label{second_sum}
	\sum_{\substack{n \in A \\ (n, P(z)) = 1}} \sum_{\substack{z \leq q < y \\ q^k \Vert n}} k
	= \sum_{\substack{n \in A \\ (n, P(z)) = 1}} \sum_{\substack{z \leq q < y \\ q \mid n}} 1
	+ \sum_{\substack{n \in A \\ (n, P(z)) = 1}} \sum_{\substack{z \leq q < y \\ q^k \Vert n,\ k \geq 2}} (k - 1).
\end{equation}
The first part in \eqref{second_sum} is given by 
\[
\sum_{\substack{n \in A \\ (n, P(z)) = 1}} \sum_{\substack{z \leq q < y \\ q \mid n}} 1
= \sum_{\substack{z \leq q < y}} S(A_q, \mathcal{P}, z).
\]
The second part in \eqref{second_sum} can be estimated as
\begin{align*}
	\sum_{\substack{n \in A \\ (n, P(z)) = 1}} \sum_{\substack{z \leq q < y \\ q^k \Vert n,\ k \geq 2}} (k - 1) 
	&\leq \sum_{z \leq q < y} \sum_{k = 2}^\infty \sum_{\substack{n \leq N' \\ q^k \mid n}} (k - 1) \\
	&\leq N' \sum_{z \leq q < y} \sum_{k = 2}^\infty \frac{k - 1}{q^k} 
	= N' \sum_{z \leq q < y} \frac{1}{(q - 1)^2} 
	< \frac{2N'}{z} = 2N'^{7/8}.
\end{align*}

\label{def_B}
For the third sum, let \( B \) be the set of integers of the form
\[
2^{k(N)} p_1 p_2 p_3 + b,
\]
where the primes \( p_1, p_2, p_3 \) satisfy
\[
z \leq p_1 < y \leq p_2 \leq p_3, \quad 2^{k(N)} p_1 p_2 p_3 \leq N.
\]
An element \( p \in B \) is prime if and only if
\[
\frac{p - b}{2^{k(N)}} = p_1 p_2 p_3 \in A, \quad \text{with } z \leq p_1 < y \leq p_2 \leq p_3.
\]

Now we compute the  third sum 
\begin{align*}
	\sum_{\substack{n \in A \\ (n, P(z)) = 1}} \sum_{\substack{p_1 p_2 p_3 = n \\ z \leq p_1 < y \leq p_2 \leq p_3}} 1
	&= \sum_{\substack{p_1 p_2 p_3 \in A \\ z \leq p_1 < y \leq p_2 \leq p_3}} 1
	= \sum_{p \in B} 1 \\
	&= \sum_{\substack{p \in B \\ p < y}} 1 + \sum_{\substack{p \in B \\ p \geq y}} 1
	\leq  y + \sum_{\substack{n \in B \\ (n, P(y)) = 1}} 1 \\
	&= N'^{1/3} + S(B, \mathcal{P}, y).
\end{align*}

Thus, we obtain the lower bound for \( n_2(N) \) in terms of sieving functions 
\begin{align}\label{sieving_functions_final_n_2}
	n_2(N) \geq \ 
	& S(A, \mathcal{P}, z)
	- \frac{1}{2} \sum_{z \leq q < y} S(A_q, \mathcal{P}, z)
	- \frac{1}{2} S(B, \mathcal{P}, y) \nonumber \\
	& - N'^{7/8} - \frac{1}{2} N'^{1/3}.
\end{align}

In the next section, we prove bounds for the sieving functions. From Theorems~\ref{lower_bound_S(A,P,z)}, \ref{upper_bound2}, and \ref{upper_bound_3}, we obtain the following estimates 
\[
S(A, \mathcal{P}, z) \geq \left( e^\gamma C_1 + O\left((\log N)^{-1/6}\right) \right) \frac{N' V(z)}{\log N},
\]
\[
\sum_{\substack{z \leq q < y}} S(A_q, \mathcal{P}, z) 
\leq \left( 4 e^\gamma C_2 + O\left((\log N)^{-1/6}\right) \right) \frac{N' V(z)}{\log N},
\]
\begin{align*}
S(B, \mathcal{P}, y) 
&\leq \left( 4\delta C_\delta e^\gamma + O\left((\log N)^{-1/6}\right) \right) \frac{N' V(z)}{\log N'} \\
&\leq \left( \frac{4\delta C_\delta e^\gamma}{1 - a} + O\left((\log N)^{-1/6}\right) \right)   \frac{N' V(z)}{\log N}.
\end{align*}

Here the constants are defined as follows 
\[
C_1 := \frac{4\delta(1 - a)}{1 - 2a} \log\left( \frac{1 - 2a}{2\delta(1 - a)} - 1 \right),
\]
\[
C_2 := \frac{\delta(1 - a)}{1 - 2a} \left( 
\log \frac{1}{3\delta} - 
\log \left( \frac{1 - \frac{2}{3} \cdot \frac{1 - a}{1 - 2a}}{1 - 2\delta \cdot \frac{1 - a}{1 - 2a}} \right)
\right),
\]
\[
C_\delta := \int_{\delta}^{1/3} \frac{\log(2 - 3\beta)}{\beta(1 - \beta)} \, d\beta.
\]

Applying these bounds to \eqref{sieving_functions_final_n_2}, we obtain 
\begin{align}\label{expression_Final}
n_2(N) \geq \left[ C_1 - 2C_2 - \frac{2\delta C_\delta}{1 - a} + O\left((\log N)^{-1/6}\right) \right] \frac{e^\gamma N' V(z)}{\log N}.
\end{align}

To ensure a positive lower bound for \( s(N, 2^{k(N)}, b) \), we seek values of \( \delta \) and \( a \) such that the main term in \eqref{expression_Final} is positive. Define the function
\begin{align}\label{f_a_d_lowerbound}
	f_{a,\delta}  
	&:= C_1 - 2C_2 - \frac{2\delta C_\delta}{1 - a} \nonumber \\
	&= \frac{4\delta(1 - a)}{1 - 2a} \log\left( \frac{1 - 2a}{2\delta(1 - a)} - 1 \right)  
	  - \frac{2\delta(1 - a)}{1 - 2a} \left( 
	\log \frac{1}{3\delta} - 
	\log \left( \frac{1 - \frac{2}{3} \cdot \frac{1 - a}{1 - 2a}}{1 - 2\delta \cdot \frac{1 - a}{1 - 2a}} \right)
	\right) 
	- \frac{2\delta C_\delta}{1 - a}.
\end{align}
We choose \( a \) and \( \delta \) to satisfy the conditions in Theorem~\ref{lower_bound_S(A,P,z)} and Theorem~\ref{upper_bound2}:
\[
\frac{1}{8} - \frac{a}{8(1 - a)} \leq \delta < \frac{1}{6} - \frac{a}{2(1 - a)}, \quad \text{and} \quad 0 < a < \frac{1}{10}.
\]
To maximize the lower bound in \eqref{f_a_d_lowerbound}, we take the smallest admissible value of \( \delta \), namely
\[
\delta = \frac{1}{8} - \frac{a}{8(1 - a)}.
\]
Setting \( a = \frac{1}{87} \), we get \( \delta = \frac{85}{688} \). A numerical evaluation yields
\[
f_{\frac{1}{87},\frac{85}{688}}  \geq 0.0001055968\ldots > 0.
\]

Hence, for sufficiently large \( N \),
\[
C := f_{\frac{1}{87},\frac{85}{688}}  + O\left((\log N)^{-1/6}\right) \geq  0.0001 > 0.
\]
Substituting \( V(z) \) via \eqref{function_V(z)} and simplifying the constants,  the constant in \eqref{expression_Final} becomes
\[
\frac{2C}{\delta} = 0.00161882 > 0.0016.
\]
 Therefore, for sufficiently large \( N \), we obtain the lower bound
\[
n_2(N) \geq 0.0016 \prod_{p > 2} \left( 1 - \frac{1}{(p - 1)^2} \right) 
\prod_{  p \mid b} \frac{p - 1}{p - 2} 
\cdot \frac{N'}{ ( \log  N')(\log N) }.
\]

	\end{proof}

Let \( s_K(N) \) denote the number of primes \( p \leq N + b \) such that \( p - b \) is divisible by \( 2^{k(N)} \) and has at most \( K \) odd prime factors. In particular, \( s_2(N) \) counts those with
\[
\frac{p - b}{2^{k(N)}} = 2^m P_2 \quad \text{for some } m \geq 0,
\]
while \( n_2(N) \) counts only those with exact divisibility,
\[
\frac{p - b}{2^{k(N)}} = P_2.
\]
To handle the broader range counted by \( s_2(N) \), we adjust the weight function. As we will show, any bound of the form
\[
n_2(N) \geq A_1 \frac{N'}{(\log N')(\log N)}
\]
implies
\[
s_2(N) \geq A_2 \frac{N'}{(\log N')(\log N)} \quad \text{with } A_2 = 2A_1,
\]
under the same conditions on \( 2^{k(N)} \).

\begin{theorem}\label{proposition_computation2}
	Let \( b > 0 \) be an odd integer, and let \( k: \mathbb{N} \to \mathbb{N} \) be an arithmetic  function such that for some constant \( 0 < c < 1 \), we have
\[
cN^a < 2^{k(N)} \leq N^a , \quad \text{where } a = \frac{1}{87}.
\]

Define \( N' := N / 2^{k(N)} \). Then for sufficiently large $N$, 
\[	s_2(N ) \geq 0.0032 \prod_{p > 2} \left(1 - \frac{1}{(p - 1)^2} \right) \prod_{2 < p \mid b} \frac{p - 1}{p - 2} \cdot \frac{N'}{(\log N)(\log N')}.\]
\end{theorem}

 \begin{proof}
We proceed similarly to the proof of Theorem~\ref{main_theorem}, with the necessary modifications.

Define
\[
N'_m := \frac{N'}{2^m}, \qquad
y^{(m)} :=   {N'_m } ^{1/3}, \qquad
M := \left\lfloor \frac{(1 - 3\delta) \log N'}{\log 2} \right\rfloor.
\]
We begin by replacing the set \( A \) from \eqref{definition_A_set} with
\[
A' = \left\{ \frac{p - b}{2^{k(N)}} : 0 < p - b \leq N,\ 2^{k(N)} \mid p - b \right\}.
\]
In the weight function \eqref{n_2_weight}, we additionally subtract contributions from representations of the form
\[
\frac{p - b}{2^{k(N)}} = 2^m p_1 p_2 p_3.
\]
If \( n = 2^m p_1 p_2 p_3 \) with \( z \leq p_1 < y \leq p_2 \leq p_3 \), then it must follow that
\[
z \leq p_1 < y^{(m)} \leq p_2 \leq p_3.
\]
We then use the modified weight function \( w'(n) \), in place of \( w(n) \) from \eqref{n_2_weight}, given by
\begin{equation}\label{w(n)_modified}
	w'(n) = 1 
	- \frac{1}{2} \sum_{\substack{z \leq q < y \\ q^k \mid\mid n}} k 
	- \frac{1}{2} \sum_{m = 0}^{M} \sum_{\substack{p_1 p_2 p_3 2^m = n \\ z \leq p_1 < y^{(m)} \leq p_2 \leq p_3}} 1,
\end{equation}
where \( y \) and \( z \) are as defined in \eqref{definition_y}. Note that \( 2^M \leq N'^{1 - 3\delta} \).
Then we have
\begin{align}\label{eq_s_2}
	s_2(N) 
	&\geq S(A', \mathcal{P}, z) 
	- \frac{1}{2} \sum_{\substack{z \leq q < y}} S(A'_q, \mathcal{P}, z) 
	- \frac{1}{2} \sum_{m = 0}^M \left[ \left( \frac{N'}{2^m} \right)^{1/3} + S(B_m, \mathcal{P}, y^{(m)}) \right]
	- N'^{7/8}\nonumber \\
		&\geq  S(A', \mathcal{P}, z) 
		- \frac{1}{2} \sum_{\substack{z \leq q < y}} S(A'_q, \mathcal{P}, z) 
		- \frac{1}{2} \sum_{m = 0}^M     S(B_m, \mathcal{P}, y^{(m)})  -N'^{1/3}
		- N'^{7/8}
\end{align}
where
\[
B_m = \left\{ 2^{k(N) + m} p_1 p_2 p_3 + b : z \leq p_1 < y^{(m)} \leq p_2 \leq p_3,\ p_1 p_2 p_3 \leq N'_m \right\}.
\]
Recall from Theorems~\ref{lower_bound_S(A,P,z)} and~\ref{upper_bound2} that
\[
S(A, \mathcal{P}, z) \geq \left( e^\gamma C_1 + O\left((\log N)^{-1/6}\right) \right) \frac{N' V(z)}{\log N},
\]
\[
\sum_{\substack{z \leq q < y}} S(A_q, \mathcal{P}, z) 
\leq \left( 4 e^\gamma C_2 + O\left((\log N)^{-1/6}\right) \right) \frac{N' V(z)}{\log N}.
\]
Since
\[
\frac{1}{\varphi(2^k)} = 2\left( \frac{1}{\varphi(2^k)} - \frac{1}{\varphi(2^{k+1})} \right),
\]
the main terms in \( |A'| \) and \( |A'_q| \), as given by the Prime Number Theorem, are twice those in \( |A| \) and \( |A_q| \), respectively.

Consequently, the bounds for \( S(A', \mathcal{P}, z) \) and \( \sum_{z \leq q < y} S(A_q', \mathcal{P}, z) \) have main terms that are twice those of the corresponding bounds for \( S(A, \mathcal{P}, z) \) and \( \sum_{z \leq q < y} S(A_q, \mathcal{P}, z) \). Thus, for \( N \) sufficiently large,
\begin{equation}\label{sum_1}
	S(A', \mathcal{P}, z) \geq \left( 2 e^\gamma C_1 + O\left((\log N)^{-1/6}\right) \right) \frac{N' V(z)}{\log N},
\end{equation} 
\begin{equation}\label{sum_2}
	\sum_{\substack{z \leq q < y}} S(A'_q, \mathcal{P}, z) 
	\leq \left( 8 e^\gamma C_2 + O\left((\log N)^{-1/6}\right) \right) \frac{N' V(z)}{\log N}.
\end{equation}

We now estimate the sum
\[
\sum_{m = 0}^M S(B_m, \mathcal{P}, y^{(m)}).
\]
By an argument analogous to Theorem~\ref{upper_bound_3}, we have
\[
S(B_m, \mathcal{P}, y^{(m)}) 
\leq \left( 4\delta e^\gamma + O\left((\log N)^{-1/6}\right) \right) \frac{C_{\delta, m} N'_m V(z)}{\log N'_m},
\]
where
\[
C_{\delta, m} := \int_{\delta \frac{\log N'}{\log N'_m}}^{1/3} \frac{\log(2 - 3\beta)}{\beta(1 - \beta)} \, d\beta.
\]

The constant \( C_{\delta, m} \) arises from evaluating the double integral (analogous to \eqref{integral}) using the change of variables \( t = N'_m{}^\alpha \) and \( u = N'_m{}^\beta \):
\begin{align*}
	\int_{z}^{y^{(m)}} \int_{y^{(m)}}^{\left( \frac{N'_m}{u} \right)^{1/2}} \frac{1}{\log(N'_m / ut)} \, d(\log \log t )\, d(\log \log u)
	&= \int_{N'^\delta}^{N'_m{}^{1/3}} \int_{N'_m{}^{1/3}}^{\left( \frac{N'_m}{u} \right)^{1/2}} \frac{1}{\log(N'_m / ut)} \, d(\log \log t )\, d(\log \log u) \\
	&= \frac{1}{\log N'_m} \int_{\delta \frac{\log N'}{\log N'_m}}^{1/3} \int_{1/3}^{\frac{1 - \beta}{2}} \frac{d\alpha \, d\beta}{\alpha \beta (1 - \alpha - \beta)} \\
	&= \frac{C_{\delta, m}}{\log N'_m}.
\end{align*}
Summing over \( m \), we obtain
\[	\sum_{m = 0}^M S(B_m, \mathcal{P}, y^{(m)}) 
\leq  \sum_{m = 0}^M \frac{C_{\delta, m} N'_m}{\log N'_m} 
\left( 4\delta e^\gamma + O\left((\log N)^{-1/6}\right) \right) V(z).\]
Compare this with the bound from Theorem~\ref{upper_bound_3}:
\begin{align}\label{compare}
S(B, \mathcal{P}, y) 
&\leq \left( 4\delta C_\delta e^\gamma + O\left((\log N)^{-1/6}\right) \right) \frac{N' V(z)}{\log N'} \nonumber\\
&\leq \left( \frac{4\delta C_\delta e^\gamma}{1 - a} + O\left((\log N)^{-1/6}\right) \right)   \frac{N' V(z)}{\log N}.
\end{align}

Define
\[
S := \sum_{m = 0}^{M} \frac{C_{\delta, m} N'_m}{\log N'_m}, \quad 
T := \frac{C_{\delta} N'}{\log N'}.
\]
We now show the following elementary lemma, which establishes that \( S/T < 2 \).

\begin{lemma}\label{lemma_S/T}
	Suppose \( \delta \geq \frac{1}{3e} \). Then we have
	\[
	S \leq 2\left(1 - \frac{1}{2^M}\right) T.
	\]
\end{lemma}

\begin{proof}
	Let \( f(t) = \frac{\log(2 - 3t)}{t(1 - t)} \), and define \( g(t) = t f(t) = \frac{\log(2 - 3t)}{1 - t} \). Let \( x = \frac{\log N'}{\log(N'_m)} \). We want to show that
	\[
	x \int_{x\delta}^{1/3} f(t)\,dt \leq \int_{\delta}^{1/3} f(t)\,dt \quad \text{for } 1 \leq x \leq \frac{1}{3\delta}.
	\]
	
	A direct computation gives
	\[
	g'(t) = \frac{-\frac{3(1 - t)}{2 - 3t} + \log(2 - 3t)}{(1 - t)^2} < 0,
	\]
	so \( g(t) \) is strictly decreasing.	
Define
\[
I(x) := \int_{x\delta}^{1/3} f(t)\,dt, \quad 
F(x) := x I(x), \quad \text{for } 1 \leq x \leq \frac{1}{3\delta}.
\]
Then
\begin{align*}
	F'(x) 
	&= I(x) + x \cdot \frac{d}{dx} I(x) 
	= I(x) - x\delta f(x\delta) \\
	&= \int_{x\delta}^{1/3} \frac{g(t)}{t}\,dt - x\delta f(x\delta) \\
	&\leq g(x\delta) \int_{x\delta}^{1/3} \frac{1}{t}\,dt - x\delta f(x\delta) \\
	&= x\delta f(x\delta) \left( \log\left(\frac{1}{3x\delta}\right) - 1 \right) \leq 0,
\end{align*}
where the last inequality uses the assumption \( \delta \geq \frac{1}{3e} \), so that \( 3x\delta \geq \frac{1}{e} \) and hence \( \log\left(\frac{1}{3x\delta}\right) \leq 1 \).

Therefore, \( F(x) \leq F(1) \) for all \( 1 \leq x \leq \frac{1}{3\delta} \), and thus
\[
\frac{C_{\delta, m}}{\log N'_m} \leq \frac{C_\delta}{\log N'} \quad \text{for all } 1 \leq m \leq M.
\]
This implies
\[
S = \sum_{m=0}^{M} \frac{C_{\delta, m} N'_m}{\log N'_m} 
\leq \frac{C_\delta}{\log N'} \sum_{m=0}^{M} N'_m 
= T \sum_{m=0}^{M} \frac{1}{2^m} 
= 2\left(1 - \frac{1}{2^{M+1}}\right) T.
\]

\end{proof}

Therefore, by Lemma~\ref{lemma_S/T} and comparison with \eqref{compare}, we obtain
\begin{equation} \label{sum_3}
	\sum_{m = 0}^M S(B_{m}, \mathcal{P}, y^{(m)}) 
	\leq  \left( \frac{8\delta  C_\delta e^\gamma}{1 - a} \left(1 - \frac{1}{2^M}\right)  + O\left((\log N)^{-1/6} \right) \right) \frac{N' V(z)}{\log N}.
\end{equation}
Recall from \eqref{expression_Final} that
\[
n_2(N) 
\geq \left(f_{a,\delta}  + O\left((\log N)^{-1/6} \right) \right) \frac{e^\gamma N' V(z)}{\log N},
\]
where
\[
f_{a,\delta}  := C_1 - 2 C_2 - \frac{2 \delta C_\delta}{1 - a}.
\]

By combining \eqref{eq_s_2}, \eqref{sum_1}, \eqref{sum_2}, and \eqref{sum_3}, we obtain
\[
s_2(N) 
\geq \left(f'_{a,\delta}  + O\left((\log N)^{-1/6} \right) \right) \frac{e^\gamma N' V(z)}{\log N},
\]
where
\[
f'_{a,\delta}  := 2 C_1 - 4 C_2 - \left(1 - \frac{1}{2^M}\right)   \frac{4 \delta C_\delta}{1 - a} > 2 f_{a,\delta} .
\]
We obtain the value of \( f'_{a,\delta}  \) by scaling \( f_{a,\delta}  \) by a factor of \( 2 \). Therefore, we adopt the same parameters as in the previous theorem. Let \( a = \frac{1}{87} \) and \( \delta = \frac{85}{688} \). Note that \( \delta = \frac{85}{688} > \frac{1}{3e} \), so the condition of Lemma~\ref{lemma_S/T} is satisfied. Then,
\[
f'_{a,\delta}  > 2f_{a,\delta} = 0.0002111936\ldots.
\]

Let $N$ be sufficiently large  such that
\[
C' = f'_{a,\delta}  + O\left((\log N)^{-1/6}\right) = 0.0002 > 0.
\]
As before, substituting \( V(z) \) using \eqref{function_V(z)} and collecting constants, the constant in the lower bound for \( s_2(N) \) becomes
\[
\frac{2C'}{\delta} = 0.00323765 > 0.0032.
\]
This completes the proof.
\end{proof}

\subsection{Bounds for the Sieving Functions}

We now derive bounds for the relevant sieving functions, beginning with an expression for \( V(z) \). By Mertens' estimate,
\begin{align}
	V(z) &= \prod_{p \mid P(z)} \left(1 - g(p)\right) \nonumber \\
	&= \prod_{q \mid b} \left(1 - \frac{1}{q - 1}\right)^{-1} \prod_{2 < p < z} \left(1 - \frac{1}{p - 1}\right) \nonumber \\
	&= 2 \prod_{q \mid b} \left(1 - \frac{1}{q - 1}\right)^{-1} 
	\prod_{p > 2} \left(1 - \frac{1}{(p - 1)^2}\right) 
	\cdot \frac{e^{-\gamma}}{\log z} \left(1 + O\left( \frac{1}{\log z} \right)\right),
	\label{function_V(z)}
\end{align}
and it follows that \eqref{condition_epsilon} is satisfied.

	Let \( k: \mathbb{N} \to \mathbb{N} \) be an arithmetic function such that for some   constant \( 0 < c < 1 \),
\begin{equation}\label{condition_2^k}
	c N^a < 2^{k(N)} \leq N^a.
\end{equation} 
Let \( A \) be the set defined in \eqref{definition_A_set}:
\[
A := \left\{ \frac{p - b}{2^{k(N)}} : 0 < p - b \leq N,\ 2^{k(N)} \,\|\, (p - b) \right\}.
\]

We now establish a lower bound for the sieving function \( S(A, \mathcal{P}, z) \).

\begin{theorem}\label{lower_bound_S(A,P,z)}
	Let \( a, \delta \) be constants satisfying
	\[
	0 < a < \frac{2}{5}, \quad \frac{1}{8} - \frac{a}{8(1-a)} \leq \delta < \frac{1}{4} - \frac{a}{4(1-a)}.
	\]
	Let \( N' = \frac{N}{2^{k(N)}} \) and \( z = N'^\delta \).
	Then for  sufficiently large \( N \), we have
	\[
	S(A, \mathcal{P}, z) \geq \left( e^\gamma C_1 + O\left((\log N)^{-1/6}\right) \right) \frac{N' V(z)}{\log N},
	\]
	where
	\[	C_{1 } := \frac{4\delta (1 - a)}{1 - 2a} \log\left(\frac{1 - 2a}{2\delta(1 - a)} - 1\right).\]
\end{theorem}

\begin{proof}
Let \( C > 0 \) be a constant to be determined later. By Theorem~\ref{guo}, we have
\begin{align*}
	|A| &= \pi(N + b, 2^{k(N)}, b) - \pi(N + b, 2^{k(N)+1}, b) \nonumber \\
	&= \frac{\operatorname{li}(N + b)}{2^{k(N)}} + O\left(\frac{N}{\varphi(2^{k(N)}) (\log N)^C}\right).  
\end{align*}
Define
\begin{equation}\label{definition_delta}
	\delta(N, q, a) := \pi(N, q, a) - \frac{\operatorname{li}(N)}{\varphi(q)}.
\end{equation}
For any \( d \mid P(z) \), we have \( (d, 2b) = 1 \), and
\[
|A_d| = \pi(N + b, 2^{k(N)} d, b) - \pi(N + b, 2^{k(N)+1} d, b),
\]
so the remainder term is given by
\[
r(d) := |A_d| - \frac{|A|}{\varphi(d)}  
= \delta(N + b, 2^{k(N)} d, b) - \delta(N + b, 2^{k(N)+1} d, b) + O\left(\frac{N}{\varphi(2^{k(N)} d)(\log N)^C}\right).
\]
Let \( B(C) \) denote the constant from Theorem~\ref{guo}, and define the level of distribution
\[
D := \frac{N^{1/2}}{2^{k(N) + 1} (\log N)^{B(C)}} 
= \frac{N'^{1/2}}{2^{k(N)/2 + 1} (\log N)^{B(C)}},
\]
where \( N' = \frac{N}{2^{k(N)}} \). We now apply the linear sieve (Theorem~\ref{Jurkat}) with \( z = N'^{\delta} \) to the set \( A \). We compute
\[
s = \frac{\log D}{\log z} 
= \frac{1}{2\delta} 
- \frac{\log 2^{k(N)/2}}{\delta \log N'} 
- \frac{\log\left(2 (\log N)^{B(C)}\right)}{\delta \log N'}.
\]
By \eqref{condition_2^k}, there exists a constant \( 0 < c' < 1 \) such that
\[
c' N'^{\frac{a}{2(1 - a)}} \leq 2^{k(N)/2} \leq N'^{\frac{a}{2(1 - a)}}.
\]
This yields the following bounds for \( s \):
\[\frac{1}{2\delta} - \frac{a}{2\delta(1-a)} - \frac{\log \left(2 (\log N)^{B(C) }\right)}{\log N'^\delta} 
\leq s \leq 
\frac{1}{2\delta} - \frac{a}{2\delta(1-a)} - \frac{\log c' + \log \left(2 (\log N)^{B(C) }\right)}{\log N'^\delta}.\]
The condition \( 2 \leq s \leq 4 \) is ensured by the inequality
\[
2 < \frac{1}{2\delta} - \frac{a}{2\delta(1-a)} \leq 4,
\]
which is equivalent to the assumption
\[
\frac{1}{8} - \frac{a}{8(1-a)} \leq \delta < \frac{1}{4} - \frac{a}{4(1-a)}.
\]

Under these conditions, we may apply the explicit form of the lower bound sieve function \( f(s) \) for \( 2 \leq s \leq 4 \), given in Lemma~\ref{initial_value}:
\[
f(s) = \frac{2e^\gamma \log(s - 1)}{s} 
= e^\gamma \cdot \frac{4\delta(1 - a)}{1 - 2a} \log\left(\frac{1 - 2a}{2\delta(1 - a)} - 1\right) + O\left(\frac{\log \log N}{\log N}\right)
= e^\gamma C_1 + O\left(\frac{\log \log N}{\log N}\right).
\]

Since \( 2^{k(N)+1} D \leq \frac{N^{1/2}}{(\log N)^{B(C)}} \), by Theorem~\ref{guo} and the bound \( \sum_{d < D} \frac{1}{\varphi(d)} \ll \log N \), the error term satisfies
\[
  \sum_{\substack{d < D \\ d \mid P(z)}} |r(d)| 
\ll \frac{N}{\varphi(2^{k(N)}) (\log N)^C} 
+ \frac{N}{\varphi(2^{k(N)}) (\log N)^C} \cdot \log N
\ll \frac{N}{\varphi(2^{k(N)}) (\log N)^{C - 1}}.
\]
Then,
\begin{align*}
	S(A, \mathcal{P}, z) 
	&\geq  \left(f(s) + O\left((\log D)^{-1/6}\right)\right) |A| V(z) + O\left(  \sum_{\substack{d < D \\ d \mid P(z)}} |r(d)| \right) \\
	&= \left(e^\gamma C_1 + O\left((\log N)^{-1/6}\right)\right) \frac{N' V(z)}{\log N}  
	+ O\left(\frac{N}{\varphi(2^{k(N)}) (\log N)^{C - 1}}\right).
\end{align*}

Taking \( C = 4 \), we obtain, for sufficiently large \( N \),
\[
S(A, \mathcal{P}, z) 
\geq \left(e^\gamma C_1  + O\left((\log N)^{-1/6}\right)\right) \frac{N' V(z)}{\log N}.
\]

\end{proof}

Next we give an estimate for the second sum in \eqref{sieving_functions_final_n_2}. 
	
	\begin{theorem}\label{upper_bound2}
Let \( a, \delta \) be constants satisfying
\[
\frac{1}{8} - \frac{a}{8(1 - a)} \leq \delta < \frac{1}{6} - \frac{a}{2(1 - a)}, \quad \text{and} \quad 0 < a < \frac{1}{10}.
\]
Let \( N' = \frac{N}{2^{k(N)}} \), \(y = N'^{1/3}\) and \( z = N'^{\delta}.\)
Then for   sufficiently large \( N \), we have
\[
\sum_{\substack{z \leq q < y}} S(A_q, \mathcal{P}, z) 
\leq \left(4 e^\gamma C_{2 } + O\left((\log N)^{-1/6}\right) \right) \frac{ N' V(z)}{  \log N},
\]
		where 
	\[	C_{2 }:= \frac{ \delta(1-a)}{1-2a} \left( \log \frac{1}{3\delta} - \log \left( \frac{1 - \frac{2}{3}\cdot \frac{ 1-a}{1-2a}  }{1- 2\delta \cdot\frac{1-a}{1-2a}}\right) \right).\]
	\end{theorem}

	\begin{proof}
Let \( A_q \) be the subset of \( A \) consisting of elements divisible by \( q \). We apply the linear sieve (Theorem~\ref{Jurkat}) to the set \( A_q \) to obtain an upper bound for 
\( S(A_q, \mathcal{P}, z) \), where \( q \) is a prime satisfying \( z \leq q < y \).

Let \( d \mid P(z) \). Since \( d \) is composed only of primes less than \( z \), and \( q \geq z \) is prime, it follows that \( (q, d) = 1 \). Therefore, the remainder term \( r_q(d) \) is given by
\begin{align*}
	r_q(d) &= |A_{q d}| - \frac{|A_q|}{\varphi(d)} \\
	&= \delta(N + b, 2^{k(N)} q d, b) - \delta(N + b, 2^{k(N) + 1} q d, b) \\
	&\quad - \frac{\delta(N + b, 2^{k(N)} q, b)}{\varphi(d)} 
	+ \frac{\delta(N + b, 2^{k(N) + 1} q, b)}{\varphi(d)},
\end{align*}
where \( \delta(N, q, a) \) is defined in \eqref{definition_delta}.
Let \( C > 0 \) be a constant to be chosen later, and define the level of distribution
\[
D_q := \frac{N^{1/2}}{2^{k(N) + 1} (\log N)^{B(C+1)} q} 
= \frac{N'^{1/2}}{2^{k(N)/2 + 1} (\log N)^{B(C+1)} q},
\]
where \( B(C+1) \) is the constant appearing in Theorem~\ref{guo}.

The error term for \( S(A_q, \mathcal{P}, z) \) is given by
\[
R_q : = \sum_{\substack{d < D_q \\ d \mid P(z)}} |r_q(d)|.
\]
Summing over all primes \( z \leq q < y \), and letting \( s_q = \frac{\log D_q}{\log z} \), we obtain
\begin{equation}\label{eq_1}
	\sum_{\substack{z \leq q < y}} S(A_q, \mathcal{P}, z)
	\leq \sum_{\substack{z \leq q < y}} \left(F(s_q) + O((\log D_q)^{-1/6})\right) |A_q| V(z) 
	+ O\left(\sum_{z \leq q < y} R_q\right).
\end{equation}

Using Theorem~\ref{guo} and the estimate \( \sum_{d \leq x} \frac{1}{\varphi(d)} \ll \log x \), we obtain
\begin{align*} 
	\sum_{z \leq q < y} R_q 
	&\ll \sum_{z \leq q < y} \sum_{\substack{d < D_q \\ d \mid P(z)}} 
	\left( \left| \delta(N + b, 2^{k(N)} q d, b) \right| + \left| \delta(N + b, 2^{k(N) + 1} q d, b) \right| \right) \nonumber \\
	&\quad + \sum_{z \leq q < y} \sum_{\substack{d < D_q \\ d \mid P(z)}} 
	\left( \frac{\left| \delta(N + b, 2^{k(N)} q, b) \right|}{\varphi(d)} 
	+ \frac{\left| \delta(N + b, 2^{k(N)+1} q, b) \right|}{\varphi(d)} \right) \nonumber \\
	&\ll \frac{N}{2^{k(N)} (\log N)^{C+1}} 
	+ \frac{N}{2^{k(N)} (\log N)^{C+1}} \sum_{d < N^{1/2}} \frac{1}{\varphi(d)} \nonumber \\
	&\ll \frac{N}{2^{k(N)} (\log N)^C}.
\end{align*}
We compute \( s_q \) as follows 
\begin{align*}
	s_q &= \frac{\log D_q}{\log z} 
	= \frac{\log \left( \frac{N'^{1/2}}{2^{k(N)/2} q} \right)}{\log N'^\delta} 
	- \frac{\log \left( 2 (\log N)^{B(C+1)} \right)}{\log N'^\delta}.
\end{align*}
As noted previously, for some constant \( 0 < c' < 1 \), we have
\[
c' N'^{\frac{a}{2(1 - a)}} \leq 2^{k(N)/2} \leq N'^{\frac{a}{2(1 - a)}}.
\]
Moreover, since \( z \leq q < y \), we obtain the following bounds 
\[
\frac{1}{6\delta} - \frac{a}{2\delta(1 - a)} - \frac{\log \left( 2 (\log N)^{B(C+1)} \right)}{\log N'^\delta}
\leq s_q \leq 
\frac{1}{2\delta} - 1 - \frac{a}{2\delta(1 - a)} - \frac{\log c' + \log \left( 2 (\log N)^{B(C+1)} \right)}{\log N'^\delta}.
\]
The condition \( 1 \leq s_q \leq 3 \) is guaranteed by the following inequalities 
\[
\frac{1}{6\delta} - \frac{a}{2\delta(1 - a)} > 1,
\qquad
\frac{1}{2\delta} - 1 - \frac{a}{2\delta(1 - a)} \leq 3,
\]
which hold under the assumptions
\[
\frac{1}{8} - \frac{a}{8(1 - a)} \leq \delta < \frac{1}{6} - \frac{a}{2(1 - a)}, 
\qquad \text{and} \qquad 0 < a < \frac{1}{10}.
\]

Therefore, we may apply the explicit form of the upper bound linear sieve function \( F(s) \) for \( 1 \leq s \leq 3 \), as given in Lemma~\ref{initial_value}:
\[
F(s_q) = \frac{2e^\gamma}{s_q} 
= \frac{2e^\gamma \log N'^\delta}{\log\left( \frac{N'^{1/2}}{2^{k(N)/2} q} \right)} 
+ O\left( \frac{\log \log N}{\log N} \right).
\]
Moreover, for each prime \( q \), we have
\begin{align*}
	|A_q| 
	&= \pi(N + b, 2^{k(N)} q, b) - \pi(N + b, 2^{k(N)+1} q, b) \\
	&= \frac{\operatorname{li}(N + b)}{2^{k(N)} \varphi(q)} 
	- \delta(N + b, 2^{k(N)} q, b) + \delta(N + b, 2^{k(N)+1} q, b).
\end{align*}
Substituting this into \eqref{eq_1}, we obtain 
\begin{align}\label{eq_3}
	&\sum_{z \leq q < y} \left(F(s_q) +  O\left((\log D_q)^{-1/6}\right)  \right) |A_q|\nonumber\\ 
	=& \sum_{z \leq q < y} \left( 
	\frac{2e^\gamma \log N'^\delta}{\log\left( \frac{N'^{1/2}}{2^{k(N)/2} q} \right)} 
	+ O\left((\log N)^{-1/6}\right)  
	\right) \frac{\operatorname{li}(N + b)}{2^{k(N)} \varphi(q)} \nonumber \\
	&\quad + O\left( \sum_{z \leq q < y} 
	\left\{ \left| \delta(N + b, 2^{k(N)} q, b) \right|  
	+ \left| \delta(N + b, 2^{k(N)+1} q, b) \right| \right\}
	\right) \nonumber \\
	=& \sum_{z \leq q < y} \left( 
	\frac{2e^\gamma \log N'^\delta}{\log\left( \frac{N'^{1/2}}{2^{k(N)/2} q} \right)} 
	+ O\left((\log N)^{-1/6}\right)  
	\right) \frac{\operatorname{li}(N + b)}{2^{k(N)}} 
	\left( \frac{1}{q} + O\left( \frac{1}{q^2} \right) \right) \nonumber \\
	&\quad + O\left( \sum_{z \leq q < y}\left\{
	\left| \delta(N + b, 2^{k(N)} q, b) \right| 
	+ \left| \delta(N + b, 2^{k(N)+1} q, b) \right| \right\}
	\right).
\end{align}	
The error terms in \eqref{eq_3} can be estimated using the following bounds
 
\[
\sum_{z \leq q < y} \frac{1}{q^2 \log\left( \frac{N'^{1/2}}{2^{k(N)/2} q} \right)} 
\ll \frac{1}{\log N'^{\frac{1}{6} - \frac{a}{2(1-a)}}} \sum_{z \leq q < y} \frac{1}{q^2} 
\ll \frac{1}{N'^\delta \log N'},
\]
\[
\sum_{z \leq q < y} \frac{1}{q} = O(1),
\]
and by Theorem~\ref{guo},
\[
\sum_{z \leq q < y} \delta(N + b, 2^{k(N)} q, b), \ 
\sum_{z \leq q < y} \delta(N + b, 2^{k(N)+1} q, b) 
\ll \frac{N}{2^{k(N)} (\log N)^{C+1}}.
\]

Now we evaluate the main term. 
Let
\[
S(t) := \sum_{q < t} \frac{1}{q} = \log \log t + d + O\left( \frac{1}{\log t} \right),
\]
for some absolute constant \( d \).
Using integration by parts, we obtain
\[
\sum_{z \leq q < y} \frac{1}{q \log\left( \frac{N'^{1/2}}{2^{k(N)/2} q} \right)} 
= \int_z^y \frac{dS(t)}{\log\left( \frac{N'^{1/2}}{2^{k(N)/2} t} \right)} 
= \int_z^y \frac{dt}{t \log t \log\left( \frac{N'^{1/2}}{2^{k(N)/2} t} \right)} 
+ O\left( \frac{1}{(\log N)^2} \right).
\]
To evaluate the integral, we make the substitution 
\[
t = \left( \frac{N'}{2^{k(N)}} \right)^\alpha.
\]
Let
\[
z' = \frac{\delta \log N'}{\log \left( \frac{N'}{2^{k(N)}} \right)}, \quad 
y' = \frac{\frac{1}{3} \log N'}{\log \left( \frac{N'}{2^{k(N)}} \right)}.
\]
Then,
\begin{align}\label{eq_4}
	\int_z^y \frac{dt}{t \log t \log\left( \frac{N'^{1/2}}{2^{k(N)/2} t} \right)} 
	&= \frac{1}{\log \left( \frac{N'}{2^{k(N)}} \right)} \int_{z'}^{y'} \frac{1}{\alpha (1/2 - \alpha)} \, d\alpha \nonumber \\
	&= \frac{2}{\log \left( \frac{N'}{2^{k(N)}} \right)} \left[ \log \alpha - \log \left( \frac{1}{2} - \alpha \right) \right]_{z'}^{y'} \nonumber \\
	&= \frac{2}{\log \left( \frac{N'}{2^{k(N)}} \right)} \left( \log \frac{1}{3\delta} 
	- \log \frac{ \frac{1}{2} - \frac{\frac{1}{3} \log N'}{\log \left( \frac{N'}{2^{k(N)}} \right)} }
	{ \frac{1}{2} - \frac{\delta \log N'}{\log \left( \frac{N'}{2^{k(N)}} \right)} } \right) \nonumber \\
	&=  {2 C_2 }   + O\left( \frac{1}{\log N } \right).
\end{align}
Taking \( C = 3 \), and combining equations~\eqref{eq_1}, \eqref{eq_3}, and \eqref{eq_4}, we obtain
\[
\sum_{\substack{z \leq q < y}} S(A_q, \mathcal{P}, z) 
\leq\left( 4e^\gamma C_{2 } + O\left((\log N)^{-1/6}\right)  \right) 
\frac{N 'V(z)}{  \log N}.
\]

	\end{proof}
In the proof of the next result, Theorem~\ref{upper_bound_3}, we will use the following bilinear form inequality.

\begin{lemma}[{\cite[Theorem 10.7]{nathanson}}]\label{bilinear_form_inequality}
	Let \( a(n) \) be an arithmetic function such that \( |a(n)| \leq 1 \) for all \( n \). Let \( A > 0 \), and suppose \( X > (\log Y)^{2A} \). Define
	\[
	D^* = \frac{(XY)^{1/2}}{(\log Y)^A}.
	\]
	Then
	\[
	\sum_{d < D^*} \max_{(a, d) = 1} \left| 
	\sum_{n < X} \sum_{\substack{Z \leq p < Y \\ np \equiv a \pmod d}} a(n) 
	- \frac{1}{\varphi(d)} \sum_{n < X} \sum_{\substack{Z \leq p < Y \\ (np, d) = 1}} a(n) 
	\right| 
	\ll \frac{XY (\log XY)^2}{(\log Y)^A},
	\]
	where the implied constant depends only on \( A \).
\end{lemma}

	\begin{theorem}\label{upper_bound_3}
		Let   \( y = N'^{1/3} \), and \( z = N'^{\delta} \) for some \( 0 < \delta < \frac{1}{3} \). Then for sufficiently large \( N' \), we have
		\[
		S(B, \mathcal{P}, y)  
	\leq  \left({4\delta e^\gamma C_\delta}  +(\log N')^{-1/6}\right) \frac{ N'V(z)}{\log N' }	,
		\]
		where
		\[
		C_\delta = \int_{\delta}^{1/3} \frac{\log(2 - 3\beta)}{\beta(1 - \beta)} \, d\beta.
		\]
		
	\end{theorem}
\begin{proof}
Recall that on p.~\pageref{def_B}, we defined
\[
B = \left\{ 2^{k(N)} p_1 p_2 p_3 + b : z \leq p_1 < y \leq p_2 \leq p_3,\ p_1 p_2 p_3 \leq N' \right\}.
\]
Let
\[
\epsilon := (\log N')^{-1/6}.
\]
Let the prime \( p_1 \) range over pairwise disjoint intervals of the form
\[
l \leq p_1 < (1+\epsilon)l,
\]
where \( l \) is of the form
\[
l = (1+\epsilon)^k z.
\]
Then
\[
0 \leq k \leq \frac{\log(y/z)}{\log(1+\epsilon)} \ll \frac{\log N'}{\epsilon}.
\]
Define
\[
B^{(l)} := \left\{ 2^{k(N)} p_1 p_2 p_3 + b : 
\begin{array}{l}
	z \leq p_1 < y  \leq p_2 \leq p_3, \\
	l \leq p_1 < (1+\epsilon)l,\ \text{and } lp_2 p_3 < N'
\end{array}
\right\},
\]
and let
\[
\tilde{B} := \bigcup_l B^{(l)}.
\]
Then
\[
B \subseteq \tilde{B} \subseteq \left\{ 2^{k(N)} p_1 p_2 p_3 + b : 
z \leq p_1 < y \leq p_2 \leq p_3,\ p_1 p_2 p_3 \leq (1+\epsilon)N' \right\}.
\]
Moreover,
\[
|\tilde{B}| = \sum_l |B^{(l)}|, \quad
S(B, \mathcal{P}, y) \leq S(\tilde{B}, \mathcal{P}, y) = \sum_l S(B^{(l)}, \mathcal{P}, y).
\]
Define the level of distribution
\[
D := \frac{N'^{1/2}}{(\log N')^6}.
\]
Applying the linear sieve (Theorem~\ref{Jurkat}) to \( B^{(l)} \) gives
\[
S(B^{(l)}, \mathcal{P}, y) 
\leq \left(F(s) + O\left((\log D)^{-1/6}\right)\right) |B^{(l)}| V(y) 
+ O\left( \sum_{\substack{d < D \\ d \mid P(y)}} |r_d^{(l)}| \right),
\]
where  
\[
r_d^{(l)} = \left| B_d^{(l)} \right| - \frac{|B^{(l)}|}{\varphi(d)},
\]
and
\[
s = \frac{\log D}{\log y} = \frac{3}{2} + O\left( \frac{\log \log N'}{\log N'} \right) \in (1, 3]. 
\]
Hence, by Lemma~\ref{initial_value},
\[
F(s) = \frac{4e^\gamma}{3} + O\left( \frac{\log \log N'}{\log N'} \right).
\]
Since
\[
\frac{V(y)}{V(z)} = 3\delta + O\left( \frac{1}{\log N'} \right),
\]
summing over \( l \) yields
\[
S(B, \mathcal{P}, y) \leq \sum_l S(B^{(l)}, \mathcal{P}, y) 
\leq \left( 4\delta e^\gamma + O\left((\log D)^{-1/6} \right) \right) |\tilde{B}| V(z) 
+ O\left( \sum_l \sum_{\substack{d < D \\ d \mid P(y)}} |r_d^{(l)}| \right).
\]

For $d\mid P(y)$, we have 
\begin{align*}
	r_d^{(l)} 
	&:= |B_d^{(l)}| - \frac{|B^{(l)}|}{\varphi(d)} \\
	&= \sum_{\substack{
			2^{k(N)} p_1 p_2 p_3 \equiv -b \pmod{d} \\
			z \leq p_1 < y \leq p_2 \leq p_3 \\
			l \leq p_1 < (1+\epsilon)l,\,lp_2 p_3 < N' \\	
	}} 1
	\ - \ \frac{1}{\varphi(d)} 
	\sum_{\substack{
			z \leq p_1 < y \leq p_2 \leq p_3 \\
			l \leq p_1 < (1+\epsilon)l \\
			lp_2 p_3 < N'
	}} 1.
\end{align*}
We can split the second sum into two cases: one where \( (p_1 p_2 p_3, d) = 1 \), which is equivalent to \( (p_1, d) = 1 \), and another where \( (p_1, d) > 1 \).
The contribution from the terms with \( (p_1, d) > 1 \) is bounded by
\begin{align*}
	\frac{1}{\varphi(d)} \sum_{\substack{p_1 p_2 p_3 \leq (1+\epsilon)N' \\ p_1 \mid d,\ p_1 \geq z \\ y < p_2 \leq N'}} 1 
	&\leq \frac{1}{\varphi(d)} \sum_{\substack{p_1 \mid d \\ p_1 \geq z}} \sum_{y < p_2 \leq N'} \pi\left( \frac{(1+\epsilon)N'}{p_1 p_2} \right) \\
	&\ll \frac{(1+\epsilon)N'}{\varphi(d)} \sum_{\substack{p_1 \mid d \\ p_1 \geq z}} \frac{1}{p_1} \\
	&\leq \frac{(1+\epsilon)N' \nu(d)}{z \varphi(d)} \\
	&\ll \frac{N'^{1 - \delta} \log d}{\varphi(d)}.
\end{align*}

Let \( a(n) \) be the characteristic function of integers \( n = p_2 p_3 \), where \( y \leq p_2 \leq p_3 \). Then,
\begin{equation}\label{eq_r_d}
	r_d^{(l)} 
	= \sum_{n < X} \sum_{\substack{Z \leq p < Y \\ np \equiv -b \cdot (2^{k(N)})^{-1} \!\!  \pmod{d}}} a(n) 
	- \frac{1}{\varphi(d)} \sum_{n < X} \sum_{\substack{Z \leq p < Y \\ (np, d) = 1}} a(n) 
	+ O\left( \frac{N'^{1 - \delta} \log d}{\varphi(d)} \right),
\end{equation}
where \( (2^{k(N)})^{-1} \) denotes the multiplicative inverse of \( 2^{k(N)} \) modulo \( d \), and
\[
X := \frac{N'}{l}, \quad 
Z := \max(z, l), \quad 
Y := \min(y, (1 + \epsilon) l).
\]
Note that since \( d \mid P(y) \), we have \( (d, 2b) = 1 \), so the inverse \( (2^{k(N)})^{-1} \bmod d \) exists and satisfies \( \left( d, b \cdot (2^{k(N)})^{-1} \right) = 1 \).
			Moreover, $$D^*   =\frac{(XY)^{1/2}}{(\log Y)^6}\geq  \frac{N'^{1/2}\min(\frac{y}{l}, 1+\epsilon )^{\frac{1}{2}}}{(\log y)^6} \geq \frac{N'^{\frac{1}{2}}}{(\log N')^6}\geq D . $$ Thus, we apply Theorem~\ref{bilinear_form_inequality} to \eqref{eq_r_d} with \( A = 6 \), obtaining
			\begin{align*}
				\sum_{\substack{d < D \\ d \mid P(y)}} \left| r_d^{(l)} \right|
				&\leq \sum_{\substack{d < D^* \\ d \mid P(y)}} \left| r_d^{(l)} \right| \\
				&\ll \frac{XY (\log XY)^2}{(\log Y)^6} 
				+ \sum_{\substack{d < D^* \\ d \mid P(y)}} \frac{N'^{1 - \delta} \log d}{\varphi(d)} \\
				&\ll \frac{N'}{(\log N')^4} + N'^{1 - \delta} (\log D^*)^2 
				\ll \frac{N'}{(\log N')^4}.
			\end{align*}	
			Therefore,
			\[
			\sum_l \sum_{\substack{d < D \\ d \mid P(y)}} \left| r_d^{(l)} \right|
			\ll \frac{\log N'}{\epsilon} \sum_{\substack{d < D \\ d \mid P(y)}} \left| r_d^{(l)} \right|
			\ll \frac{N'}{(\log N')^{3 - 1/6}}.
			\]
		
We now estimate \( |\tilde{B}| \). Observe that
\begin{align*}
	|\tilde{B}| 
	&= \sum_{\substack{z \leq p_1 < y \leq p_2 \leq p_3 \\ p_1 p_2 p_3 \leq (1 + \epsilon) N'}} 1 \\
	&\leq  \sum_{\substack{z \leq p_1 < y \leq p_2 \\ p_1 p_2^2 \leq (1 + \epsilon) N'}} \pi\left( \frac{(1 + \epsilon) N'}{p_1 p_2} \right) \\
	&\leq (1 + 2\epsilon) N' \sum_{z \leq p_1 < y} \frac{1}{p_1} \sum_{y \leq p_2 < \left( \frac{(1 + \epsilon) N'}{p_1} \right)^{1/2}} \frac{1}{p_2 \log \left( \frac{N'}{p_1 p_2} \right)}.
\end{align*}
This double sum can be approximated by the double integral
\begin{align*}
	\sum_{z \leq p_1 < y} \frac{1}{p_1} \sum_{y \leq p_2 < \left( \frac{(1+\epsilon)N'}{p_1} \right)^{1/2}} \frac{1}{p_2 \log\left( \frac{N'}{p_1 p_2} \right)} 
	= \int_{z}^{y} \int_{y}^{\left( \frac{N'}{u} \right)^{1/2}} \frac{1}{\log(N' / ut)} \, d(\log \log t) \, d(\log \log u) 
	+ O\left( \frac{1}{(\log N')^2} \right).
\end{align*}
Using the change of variables \( t = N'^{\alpha} \), \( u = N'^{\beta} \), the integral becomes
\begin{align}\label{integral}
	\int_{z}^{y} \int_{y}^{\left( \frac{N'}{u} \right)^{1/2}} \frac{1}{\log(N' / ut)} \, d(\log \log t) \, d(\log \log u )
	&= \int_{N'^{\delta}}^{N'^{1/3}} \int_{N'^{1/3}}^{\left( \frac{N'}{u} \right)^{1/2}} \frac{1}{\log(N' / ut)} \, d(\log \log t )\, d(\log \log u) \nonumber \\
	&= \frac{1}{\log N'} \int_{\delta}^{1/3} \int_{1/3}^{\frac{1 - \beta}{2}} \frac{d\alpha \, d\beta}{\alpha \beta (1 - \alpha - \beta)} \nonumber \\
	&= \frac{C_\delta}{\log N'},
\end{align}
where
\[
C_\delta := \int_{\delta}^{1/3} \frac{\log(2 - 3\beta)}{\beta(1 - \beta)} \, d\beta.
\]		

Therefore,
\[
|\tilde{B}| \leq \frac{C_\delta N'}{\log N'} (1 + 2\epsilon) + O\left( \frac{N'}{(\log N')^2} \right),
\]
and hence, for sufficiently large $N$, 
\begin{align*}
	S(B, \mathcal{P}, y) 
	&\leq \left( 4\delta e^\gamma + O\left( (\log D)^{-1/6} \right) \right) |\tilde{B}| V(z) 
	+ O\left( \frac{N'}{(\log N')^{3 - 1/6}} \right) \\
	&\leq \left( 4\delta e^\gamma C_\delta + O\left( (\log N')^{-1/6} \right) \right) \frac{N' V(z)}{\log N'}.
\end{align*}

\end{proof}

\section{Application of Richert's Weight: \( p - b \) Divisible by Larger Powers of \( 2 \)}\label{section_weighted_sieve}

The constraint on the upper bound for \( 2^{k(N)} \) in the previous section arises from the choice of parameters \( a \) and \( \delta \) in the linear sieve with Chen's weight. In this section, we apply Richert's weights in the linear sieve framework \cite[Theorem 9.1]{halberstam}, which allows us to accommodate larger values of \( 2^{k(N)} \) while permitting more odd prime divisors.

We follow the weighted sieve formulation of Halberstam and Richert \cite[Chapter~9]{halberstam}, which provides explicit expressions for the lower bound coefficients in the linear sieve weight function. This allows us to apply the result directly, without requiring    numerical computations.

We begin with a version of \cite[Lemma 3.5]{halberstam} adapted to moduli with large power factors, where the Bombieri--Vinogradov theorem is replaced by Theorem~\ref{guo}.
 
\begin{lemma}\label{Lemma_guo}
	Let \( h > 0 \) and \( a \geq 2 \) be integers. Suppose \( q \) is a power of \( a \) satisfying
	\begin{equation}\label{condition_q}
		q \leq x^{2/5} \exp\left( -(\log \log x)^3 \right).
	\end{equation}
	Define
	\begin{equation}\label{def_E}
			E(x, d) := \max_{2 \leq y \leq x} \max_{(l,d) = 1} \left| \pi(y, d, l) - \frac{\operatorname{li}(y)}{\varphi(d)} \right|.
	\end{equation}
	Then, for any positive constant \( U \), there exists a constant \( C_1 = C_1(U, h) > 0 \) such that
	\[
	\sum_{d \leq \frac{x^{1/2}}{q (\log x)^{C_1}}} \mu^2(d) h^{2\nu(d)} E(x, dq)
	= O_{U, h} \left( \frac{x}{\varphi(q) (\log x)^U} \right).
	\]
\end{lemma}

\begin{proof}
	The argument follows that of \cite[Lemma 3.5]{halberstam}, replacing the condition \( q \leq (\log x)^A \) with the bound \eqref{condition_q}, and using Theorem~\ref{guo} in place of the Bombieri--Vinogradov theorem.
\end{proof}

\begin{theorem}
Let \( b > 0 \) be an odd integer, and let \( \epsilon > 0 \) be a small fixed constant. Suppose \( k: \mathbb{N} \to \mathbb{N} \) is an arithmetic function satisfying
\[
2^{k(N)} \leq N^{a - \epsilon}, \quad \text{where } a = \frac{1}{6}.
\]

Let \( s_3(N  ) \) denote the number of primes \( p \leq N + b \) such that \( p - b \) is divisible by \( 2^{k(N)} \) and has at most three odd prime factors.

Define \( N' := N / 2^{k(N)} \). Then for sufficiently large $N$,  
\[	s_3(N   ) \geq \frac{40\theta}{\left( \frac{2}{5} + 6\theta \right)^2} \log 3 \cdot \prod_{p > 2} \left(1 - \frac{1}{(p - 1)^2} \right) \prod_{2 < p \mid b} \frac{p - 1}{p - 2} \cdot \frac{N'}{(\log N)(\log N')},\]
where
\[
\theta = \frac{\epsilon}{2(1 - a + \epsilon)(1 - a)}.
\]

\end{theorem}
\begin{proof}
The proof follows the same general strategy as the one for the infinitude of primes \( p \) such that \( p - b \) has at most three prime factors, as presented in Halberstam and Richert~\cite[Theorem~9.2]{halberstam}, which uses a linear weighted sieve~\cite[Theorem~9.1]{halberstam}. We adopt the same notation as in~\cite{halberstam} throughout.

Let
\[
A = \left\{ \frac{p - b}{2^{k(N)}} : 0 < p - b \leq N,\ p \text{ prime},\ p \equiv b \pmod{2^{k(N)}} \right\},
\]
and define
\[
\mathcal{B}_{2b} := \left\{ q : q \text{ prime},\ q > 2,\ (q, b) = 1 \right\}.
\]
We set
\[
X = \frac{\operatorname{li}(N + b)}{\varphi(2^{k(N)})}, \quad \omega(p) = \frac{p}{p - 1} \quad \text{for } p \nmid 2b.
\]
Then, for 
\[
R_d := |A_d| - \frac{\omega(d)}{d} X,
\]
we have the bound
\[
|R_d| \leq E(N + b, 2^{k(N)} d),
\]
where \( E(N + b, 2^{k(N)} d) \) is as defined in \eqref{def_E}.

The conditions \((\Omega_1)\) \cite[p.~29]{halberstam} and \((\Omega_2(1, L))\) \cite[p.~142]{halberstam} are satisfied with
\[
L \leq O(1) + \sum_{p \mid 2b} \frac{\log p}{p - 1} \ll \log \log 6b \ll 1,
\]
as follows from Mertens' estimates; see also \cite[(1.1), (1.5), p.~143]{halberstam}.

For the condition \((R(1, \alpha))\) \cite[p.~236]{halberstam}, Lemma~\ref{Lemma_guo} implies that for a suitable constant \( C_1 \),
\[
\sum_{d \leq \frac{N^{1/2}}{2^{k(N)} (\log N)^{C_1}}} \mu^2(d)\, 3^{\nu(d)}\, R_d 
= O\left( \frac{N}{\varphi(2^{k(N)}) (\log N)^3} \right) 
= O\left( \frac{N'}{(\log N)^3} \right) 
= O\left( \frac{X}{(\log X)^2} \right).
\]
The upper bound on \( d \) in the summation satisfies
\[
\frac{N^{1/2}}{2^{k(N)} (\log N)^{C_1}} = \frac{N'^{1/2}}{2^{k(N)/2} (\log N)^{C_1}} 
\geq \frac{N'^{\frac{1}{2} - \frac{a}{2(1 - a)} + \theta}}{(\log N)^{C_1}},
\]
where
\[
\theta = \frac{\epsilon}{2(1 - a + \epsilon)(1 - a)},
\]
so that the condition \((R(1, \alpha))\) is satisfied with
\begin{equation}\label{condition_alpha}
	\alpha = \frac{1}{2} - \frac{a}{2(1 - a)} + \theta.
\end{equation}

Define the weighted sum  
\begin{equation}\label{def_weighted}
	W(A; \mathcal{B}_{2b}, v, u, \lambda) := 
	\sum_{\substack{y \in A \\ (y, P(X^{1/v})) = 1}} 
	\left( 1 - \lambda \sum_{\substack{X^{1/v} \leq p < X^{1/u} \\ p \mid y,\ p \in B_{2b}}} 
	\left( 1 - u \frac{\log p}{\log X} \right) \right).
\end{equation}

Taking the parameters
\begin{equation}\label{parameters}
	\alpha v = 4, \quad \alpha u = \frac{4}{3},
\end{equation}
and applying \cite[Theorem~9.1 and Lemma~9.1]{halberstam}, we obtain the estimate
\begin{equation}\label{weight_equation}
	W(A; \mathcal{B}_{2b}, v, u, \lambda) 
	\geq X W(X^{1/v}) \left( \frac{e^\gamma}{2} \left(1 - \frac{2}{3} \lambda \right) \log 3 
	- \frac{B_1}{(\log X)^{1/15}} \right),
\end{equation}
where \( W(z) \) is defined as
\[
W(z) := \prod_{p < z} \left(1 - \frac{\omega(p)}{p} \right),
\]
and \( B_1 \) is an absolute constant.

By Mertens’ estimate,  we express \( W(X^{1/v}) \) as
\begin{align*}
	W(X^{1/v}) = 
	2v \prod_{p > 2} \left( 1 - \frac{1}{(p - 1)^2} \right) 
	\prod_{2 < p \mid b} \frac{p - 1}{p - 2} \cdot \frac{e^{-\gamma}}{\log X} 
	\left( 1 + O\left( \frac{1}{\log X} \right) \right).
\end{align*}
Substituting into \eqref{weight_equation}, we obtain
\[W(A; \mathcal{B}_{2b}, v, u, \lambda) 
\geq   v \prod_{p > 2} \left( 1 - \frac{1}{(p - 1)^2} \right)
\prod_{2 < p \mid b} \frac{p - 1}{p - 2} \cdot \frac{X}{\log X}  \cdot
\left[ \left( 1 - \frac{2}{3} \lambda \right) \log 3 
- \frac{B_2}{(\log X)^{1/15}} \right],\]
where \( B_2 \) is an absolute constant.
Since
\[
X = \frac{\operatorname{li}(N + b)}{\varphi(2^{k(N)})} \sim \frac{2N'}{\log N},
\]
we can further write
\begin{align}\label{weight_final_form_2}
	W(A; \mathcal{B}_{2b}, v, u, \lambda) 
	\geq {} & v \prod_{p > 2} \left( 1 - \frac{1}{(p - 1)^2} \right)
	\prod_{2 < p \mid b} \frac{p - 1}{p - 2} \cdot \frac{2N'}{(\log N)(\log N')} \nonumber \\
	& \cdot \left[ \left( 1 - \frac{2}{3} \lambda \right) \log 3 
	- \frac{B_{3}}{(\log N')^{1/15}} \right],
\end{align}
where \( B_3 \) is an absolute constant.

To establish a result for elements with at most \( K \) distinct odd prime factors, we require that elements \( y \in S(A, X^{1/v}) \) with more than \( K \) distinct odd prime divisors contribute non-positively to the weight function \eqref{def_weighted}.
 This leads to the condition 
\begin{equation}\label{condition_K}
	1 - \lambda \left( K + 1 - u \cdot \frac{\log N'}{\log X} \right) \leq 0.
\end{equation}

Moreover, as shown in \cite[p.~250--251]{halberstam}, the contribution from elements \( y \in S(A, X^{1/v}) \) divisible by \( p^2 \), where \( X^{1/v} \leq p < X^{1/u} \), can be sufficiently bounded. Elements \( a \) divisible by \( p^2 \) with \( p \geq X^{1/u} \) always contribute negatively to the weights \( 1 - u \frac{\log p}{\log X} \). Therefore, if we additionally count the number of prime divisors of \( a \) with multiplicity, the corresponding term in \eqref{def_weighted} can only increase and is at most
\[
1 - \lambda \left( \Omega(a) + 1 - u \cdot \frac{\log |a|}{\log X} \right).
\]
Hence, the difference between counting prime factors with or without multiplicity is negligible, and the same condition \eqref{condition_K} ensures the validity of the result for elements with at most \( K \) (not necessarily distinct) odd prime factors.

To ensure that \eqref{condition_K} holds, it suffices to require
\begin{equation}\label{condition_lambda_u}
	1 - \lambda (K + 1 - u) < 0.
\end{equation}

Furthermore, to guarantee a positive lower bound in \eqref{weight_final_form_2}, we must also have 
\begin{equation}\label{condition_lambda}
	1 - \frac{2}{3} \lambda > 0.
\end{equation}

By \eqref{condition_alpha}, \eqref{parameters}, \eqref{condition_lambda_u}, and \eqref{condition_lambda}, we obtain the condition
\[
1 - \frac{1}{\frac{6K - 12}{3K - 2} + \theta} < a < 1 - \frac{1}{\frac{6K - 6}{3K + 1} + \theta}.
\]
It suffices to choose a value of \( a \) that satisfies this inequality for all sufficiently small \( \theta > 0 \), while also maximizing \( a \). We take
\begin{equation}\label{relation_a_and_K}
	a = 1 - \frac{1}{\frac{6K - 6}{3K + 1}} = \frac{3K - 7}{6K - 6}.
\end{equation}

Fixing \( K = 3 \), we have \( a = \frac{1}{6} \), so \( \alpha = \frac{2}{5} + \theta \). The parameters in \eqref{parameters} then become
\[
v = \frac{4}{\alpha} = \frac{4}{\frac{2}{5} + \theta}, \quad 
u = \frac{4}{3\alpha} = \frac{4}{3\left(\frac{2}{5} + \theta\right)}.
\]
With these choices of \( u \) and \( v \), and for any \( \lambda \) satisfying
\[
  \lambda > \frac{1}{K + 1 - u} = \frac{6/5 + 3\theta}{4/5 + 12\theta},
\]  we obtain the lower bound
\[	s_3(N) 
\geq W(A; \mathcal{B}_{2b}, v, u, \lambda) + O\left( N'^{1 - \frac{1}{v}} \log N' \right),\]
where the error term arises from bounding the contribution of elements divisible by \( p^2 \) for \( X^{1/v} \leq p < X^{1/u} \); see, for example, equation (2.6) in \cite[Theorem~9.2]{halberstam}.

When
\[
\lambda = \frac{1}{K + 1 - u} = \frac{6/5 + 3\theta}{4/5 + 12\theta},
\]
the constant in the main term of \eqref{weight_final_form_2} becomes
\[
2v\left(1 - \frac{2}{3} \lambda\right) \log 3 
= \frac{40\theta}{\left(\frac{2}{5} + \theta\right)\left(\frac{2}{5} + 6\theta\right)} \log 3 
> \frac{40\theta}{\left(\frac{2}{5} + 6\theta\right)^2} \log 3>0. 
\]
 Therefore, by choosing a slightly larger \( \lambda > \frac{6/5 + 3\theta}{4/5 + 12\theta} \) so that   \[2v\left(1 - \frac{2}{3} \lambda\right) \log 3 = 
 \frac{40\theta}{\left(\frac{2}{5} + 6\theta\right)^2} \log 3  ,\] we have  the desired lower bound.
\end{proof}
\begin{remark}
	A similar result holds for \( n_3(N) \), counting primes \( p \) with \( p - b \leq N \) and  \(  p - b =2^{k(N)} P_3 \). This can be achieved   by replacing the sifting set with
	\[
	A = \left\{ 
	\frac{p - b}{2^{k(N)}} : p - b \leq N,\ 2^{k(N)} \mid\mid (p - b)
	\right\}.
	\]
\end{remark}

Let \( s_K(N  ) \) denote the number of primes \( p \leq N + b \) such that \( p - b \) is divisible by \( 2^{k(N)} \) and has at most $K$ odd prime factors.

The relation \eqref{relation_a_and_K} applies for all \( K \) such that \( a \leq \frac{2}{5} \), as required by Lemma~\ref{Lemma_guo}. We thus obtain the following.

\begin{cor}
Let \( b > 0 \) be an odd integer, and let \( \epsilon > 0 \) be a small fixed constant. Suppose \( k: \mathbb{N} \to \mathbb{N} \) is an arithmetic function satisfying
\[
2^{k(N)} \leq N^{a - \epsilon}, \quad \text{for some } 0<a <1.
\]

Define \( N' := N / 2^{k(N)} \). Then for \( K = 4, 5, 6, 7 \), and all sufficiently large \( N \), we have
\begin{equation}\label{s_7}
	s_K(N  ) \geq \frac{8\theta(1 + 3K)^3\log 3 }{\left[8 + 3\theta(1 + 4K + 3K^2)\right]^2}   \prod_{p > 2} \left(1 - \frac{1}{(p - 1)^2}\right) \prod_{2 < p \mid b} \frac{p - 1}{p - 2} \cdot \frac{N'}{(\log N)(\log N')},
\end{equation}
where
\[
\theta = \frac{\epsilon}{2(1 - a + \epsilon)(1 - a)}, \quad \text{and} \quad a = \frac{3K - 7}{6K - 6}.
\]
The corresponding values of \( a \) for each \( K \) are given below:

\begin{center}
	\begin{tabular}{|c|c|c|c|c|}
		\hline
		\( K \) & 4 & 5 & 6 & 7 \\
		\hline
		\( a \) & \( \dfrac{5}{18} \) & \( \dfrac{1}{3} \) & \( \dfrac{11}{30} \) & \( \dfrac{7}{18} \) \\
		\hline
	\end{tabular}
\end{center}

For \( K = 8 \), suppose instead that
\[
2^{k(N)} \leq N^{2/5} \exp\left(-(\log \log N)^3\right),
\]
then for all sufficiently large \( N \), we have 
\begin{equation}\label{s_8_weighted_sieve_result}
	s_8(N  ) \geq \frac{52}{3} \cdot \prod_{p > 2} \left(1 - \frac{1}{(p - 1)^2}\right) \prod_{2 < p \mid b} \frac{p - 1}{p - 2} \cdot \frac{N'}{(\log N)(\log N')}.
\end{equation}
\end{cor}

\begin{proof}
This follows from the previous theorem by adjusting the sieve parameters. Let
\[
a = \frac{3K - 7}{6K - 6}, \quad \text{so that} \quad \alpha = \frac{4}{3K + 1} + \theta,
\]
where
\[
\theta = \frac{\epsilon}{2(1 - a + \epsilon)(1 - a)}.
\]
We seek a suitable \( \lambda \) satisfying
\[
  \lambda > \frac{1}{K + 1 - u} = \frac{1}{K + 1 - \frac{4}{3\alpha}}.
\]
If we take equality  \( \lambda = \frac{1}{K + 1 - \frac{4}{3\alpha}} \), then the constant in the main term of \eqref{weight_final_form_2} becomes
\[
2v \left(1 - \frac{2}{3} \lambda \right) \log 3
= \frac{8\theta(1 + 3K)^3}{(4 + \theta + 3\theta K)(8 + 3\theta(1 + 4K + 3K^2))} \log 3
> \frac{8\theta(1 + 3K)^3}{(8 + 3\theta(1 + 4K + 3K^2))^2} \log 3.
\]
Therefore, by choosing a slightly larger value \( \lambda > \frac{1}{K + 1 - \frac{4}{3\alpha}} \) so that
\[
2v \left(1 - \frac{2}{3} \lambda \right) \log 3 
= \frac{8\theta(1 + 3K)^3}{(8 + 3\theta(1 + 4K + 3K^2))^2} \log 3,
\]
we obtain the desired lower bound \eqref{s_7}, valid for \( K \leq 7 \), where \( a = \frac{3K - 7}{6K - 6} < \frac{2}{5} \).

For \( K = 8 \), we have \( \frac{3K - 7}{6K - 6} > \frac{2}{5} \), so instead we take
\[
2^{k(N)} \leq N^{2/5} \exp\left( -(\log \log N)^3 \right),
\]
as required for the power modulus upper bound in Lemma~\ref{Lemma_guo}. In this case, the condition \( (R(1, \alpha)) \) is satisfied with
\[
\alpha = \frac{1}{6}.
\]
We now seek a parameter \( \lambda \) satisfying
\[
 \lambda > \frac{1}{K + 1 - \frac{4}{3\alpha}} = 1.
\]
When \( \lambda = 1 \), the constant in the main term of \eqref{weight_final_form_2} becomes
\[
2v \left(1 - \frac{2}{3} \lambda \right) \log 3 = 16 \log 3 > \frac{52}{3}.
\]
Hence, by choosing a slightly larger \( \lambda > 1 \) such that
\[
2v \left(1 - \frac{2}{3} \lambda \right) \log 3 = \frac{52}{3},
\]
we obtain the desired lower bound \eqref{s_8_weighted_sieve_result}.
\end{proof}

	\bibliographystyle{plain}

\begin{thebibliography}{99}
	
\bibitem{chen} Chen, J.-R. (1973). On the representation of a large even integer as the sum of a prime and the product of at most two primes. \textit{Scientia Sinica}, 16(2), 157–176.
	
	\bibitem{prime_power} Elliott, P. D. T. A. (2007). Primes in progressions to moduli with a large power factor. \textit{The Ramanujan Journal}, 13, 241-251.
	
	
 \bibitem{opera} Friedlander, J. B., \& Iwaniec, H. (2010). \textit{Opera de cribro} (Vol. 57). American Mathematical Soc..
	
	\bibitem{halberstam} Halberstam, H., \& Richert, H.-E. (2013). \textit{Sieve methods}. Dover Publications.
	
	\bibitem{heathbrown} Heath‐Brown, D. R. (1992). Zero‐Free Regions for Dirichlet L‐Functions, and the Least Prime in an Arithmetic Progression. \textit{Proceedings of the London Mathematical Society}, 3(2), 265-338.
	
	\bibitem{hooley} Hooley, C. (1976). Applications of sieve methods to the theory of numbers. \textit{Cambridge University Press}.
	
	\bibitem{iwaniec} Iwaniec, H., \& Kowalski, E. (2021). \textit{Analytic number theory (Vol. 53).} American Mathematical Society.
	
	\bibitem{jurkat-Richert} Jurkat, W., \& Richert, H. (1965). An improvement of Selberg's sieve method I. \textit{Acta Arithmetica, 11}(2), 217-240.
	\bibitem{jutlia} Jutila, M. (1972). On a density theorem of HL Montgomery for L-functions. \textit{Annales Fennici Mathematici}, (520).
	
	\bibitem{montgomery} Montgomery, H. L. (2006). \textit{Topics in multiplicative number theory} (Vol. 227). Springer.
	

	
	
	\bibitem{nathanson} Nathanson, M. B. (2013). \textit{Additive Number Theory: The Classical Bases (Vol. 164)}. Springer Science \& Business Media.
	
	
\end{thebibliography}

\end{document}